\documentclass[sn-mathphys,Numbered]{sn-jnl}

\usepackage{multirow}%
\usepackage{amsmath,amssymb,amsfonts}%
\usepackage{amsthm}%
\usepackage{mathrsfs}%
\usepackage[title]{appendix}%
\usepackage{xcolor}%
\usepackage{textcomp}%
\usepackage{manyfoot}%
\usepackage{booktabs}
\usepackage{algorithm}%
\usepackage{algorithmicx}%
\usepackage{algpseudocode}%
\usepackage{listings}%

\theoremstyle{thmstyleone}%
\newtheorem{theorem}{Theorem}

\theoremstyle{thmstyletwo}%
\newtheorem{corollary}{Corollary}%
\newtheorem{lemma}{Lemma}%

\theoremstyle{thmstylethree}%

\def \ve{\epsilon}

\def\Re{{\rm I\kern-0.2em R}}
\def\R{\Re}

\newcommand{\brou}[1]{\bigl( #1 \bigr)}

\begin{document}

\title[On Averaging of a Class of  \lq\lq Non-Singularly Perturbed"  Control Systems]{On Averaging of a Class of \lq\lq  Non-Singularly Perturbed"  Control Systems}

\author[1]{\fnm{Vladimir} \sur{Gaitsgory}}\email{vladimir.gaitsgory@mq.edu.au}

\author*[2]{\fnm{Ilya} \sur{Shvartsman}}\email{ius13@psu.edu}

\affil[1]{\orgdiv{Department of Mathematics}, \orgname{Macquarie University}, \orgaddress{\city{Sydney}, \postcode{2109}, \state{NSW}, \country{Australia}}}

\affil*[2]{\orgdiv{Department of Computer Science and Mathematics}, \orgname{Penn State Harrisburg}, \orgaddress{\city{Middletown}, \postcode{17057}, \state{PA}, \country{USA}}}

\abstract{We study a control system resembling a singularly perturbed system whose variables are decomposed into groups that change their values with rates of different orders of magnitude. We establish that the   slow  trajectories of this system are dense in the set of solutions of a certain differential inclusion and discuss an implication of this result for optimal control.}

\keywords{singular perturbations, averaging method, optimal control
}

\pacs[AMS Subject Classification]{34E15, 34C29, 34A60}

\maketitle

\section{Introduction and the main result}\label{S1}

In this paper, we consider the  system
\begin{eqnarray}\label{e:fast}
\ve \frac{dy(t)}{dt} &=&Ay(t)+Bu(t), \ \ \ \ y(0)=y_0,     \\
\frac{dz(t)}{dt}&=&g(u(t),y(t),z(t))  ,   \ \ \ \ z(0)=z_0,   \label{e:slow}
\end{eqnarray}
where $\ve > 0$ is a parameter, the controls $u(t)$ are measurable functions taking values in $\R^k$, and
$A$ and $B$ are $m\times m$ and $m\times k$ matrices that satisfy the rank controllability condition:
\begin{equation}\label{e-rank}
{\rm rank} [ B, AB,..., A^{m-1}B]=m.
\end{equation}
The function
 $g:\R^k\times \R^m \times
\R^n\to \R^n$ is assumed to be bounded
\begin{equation}\label{e-g-bound}
\sup_{(u,y,z)\in\R^k\times\R^m\times\R^n}||g(u,y,z)||:= M_g<\infty
\end{equation}
  and satisfying  the Lipschitz condition in
$(y,z)$ uniformly with respect to $u\in \R^k$.

If the parameter $\ve$ was small, the system (\ref{e:fast})-(\ref{e:slow}) would belong to the class of the so-called {\em singularly perturbed} (SP) systems.
SP systems are characterized by the decomposition of the state variables into groups that change their values with rates of different orders of magnitude (so that some of them can be considered as fast/slow with respect to others), which is due to the presence of the small singular perturbations parameter $\ve$. Such systems describe processes and interactions in disparate time scales, and they
have been extensively studied in the literature (see, e.g.,
research monographs  \cite{Ben}, \cite{Kab2},
 \cite{Kok1}, \cite{Kus3} and
  surveys  \cite{Dmitriev},  \cite{Nai}, \cite{Reo1}, \cite{Zhang-Naidu}).

One of the approaches to  SP control systems is the  {\em averaging method} based on the analysis of asymptotic properties of the sets of time averages of the equation describing
the   dynamics of the slow state variables over the fast   control-state trajectories considered on the intervals $[0,S]$ for large values of $S>0$. If the limit of these sets exists as $S$ tends to infinity, then (under certain additional conditions) it defines the right-hand-side of the differential inclusion, the solutions of which approximate the slow components of the solutions of the SP system when the singular perturbation parameter $\ve$ tends to zero (see \cite{Dont-Don},    \cite{Gai0}, \cite{Gai1}, \cite{Gra},
  \cite{QW} and also  \cite{Alv}, \cite{Art2}, \cite{AG} for related developments).

In this paper,  the parameter $\ve$  is not assumed to be small, and to distinguish it from the case of singular perturbations, we call system
(\ref{e:fast})-(\ref{e:slow}) {\em  non-singularly perturbed}. 
We show that the approximation of the $z$-components of the solution
of system (\ref{e:fast})-(\ref{e:slow}) by the solution of a certain differential inclusion is still valid, but this result is no longer  asymptotic and
holds for any $\ve>0$. Such an approximation is possible due to the fact that  the $y$-components of the state variables of the system (\ref{e:fast})-(\ref{e:slow}) can change their values arbitrarily fast (see Lemma \ref{Lem-control} in the next section) while the rates of change of the $z$-components are bounded by the constant $M_g$ (see
(\ref{e-g-bound})).

To state our main result, let us introduce the differential inclusion
\begin{equation}\label{eq-DI}
\frac{dz(t)}{dt}\in V(z(t)) ,\ \ \ \ z(0)=z_0,
\end{equation}
where
\begin{equation}\label{eq-RHS}
V(z):= \bar{\rm co} (g(\R^k, \R^m, z)), \ \ \ \ g(\R^k, \R^m, z):= \{v\ | \ v =g(u,y,z), \ u\in\R^k, \ y\in \R^m\},
\end{equation}
with $\bar{co} $ standing for the closure of the convex hull of the corresponding set.

Let $T$ be an arbitrary positive number. Denote by $\mathcal{Z}_T(\ve)$  the set of the $z$-components of  solutions of the system
(\ref{e:fast})-(\ref{e:slow}) considered on the interval $[0,T]$, that is,
$$
\mathcal{Z}_T(\ve):=\{z(\cdot)\,|\, (u(\cdot),y(\cdot),z(\cdot))\hbox{ satisfies }(\ref{e:fast}),(\ref{e:slow})\}.
$$
By $\mathcal{Z}_T$ denote the set of solutions of the differential inclusion (\ref{eq-DI})
considered on this interval. Note that, as can be readily understood,
\begin{equation}\label{eq-INCL}
\mathcal{Z}_T(\ve)\subset \mathcal{Z}_T \ \ \ \ \forall \ \ve >0.
\end{equation}
The main result of the paper is the following theorem.

\begin{theorem}\label{Th-main}
The equality
\begin{equation}\label{eq-EQ}
{\rm cl}(\mathcal{Z}_T(\ve))= \mathcal{Z}_T \ \ \ \ \forall \ \ve >0
\end{equation}
is valid, where ${\rm cl}$ in the expression above stands for the closure in the uniform  convergence metrics.
That is,
the set  of the $z$-components of  solutions of the system
(\ref{e:fast})-(\ref{e:slow}) is dense in the set  of solutions of the differential inclusion (\ref{eq-DI}).
\end{theorem}
\begin{proof}
The proof of the theorem is given in the next section.
\end{proof}

REMARK I. Theorem \ref{Th-main} resembles results establishing that the Hausdorff distance (induced by the uniform  convergence metric) between the set of the slow components of  solutions of a SP control system and the set  of solutions of the differential inclusion, constructed by averaging of the slow subsystem over the controls and the corresponding solutions of the fast one, tends to zero when the   singular perturbations parameter tends to zero (see
 \cite{Dont-Don}, \cite{Gai0},  \cite{Gai1}, \cite{Gra},
  \cite{QW}). As mentioned above,  in contrast to these results, Theorem \ref{Th-main} is not of asymptotic nature. The equality (\ref{eq-EQ}) is valid 
  for any $\ve >0$, including, e.g., $\ve=1$.

 REMARK II. Note that  the statement of Theorem \ref{Th-main} would look similar to that of Filippov-Wazewski theorem (see \cite{aubin1})
  if $y(t)$ in (\ref{e:slow})  was another control (that is, an arbitrary measurable function taking values in $\R^m$) instead of being the solution of (\ref{e:fast}).

Let us discuss an implication of Theorem \ref{Th-main} for optimal control. Consider the optimal control problem
 \begin{equation}\label{e-SP-problem}
 \inf_{u(\cdot)}G(z_{\ve}(T)):= G_{\ve}^*,
 \end{equation}
 where $G(\cdot)$ is a continuous function and $inf$ is taken over the controls $u(\cdot)$ and the corresponding solutions $(y_{\ve}(\cdot),z_{\ve}(\cdot))$ of system (\ref{e:fast})-(\ref{e:slow}).
Consider also the optimal control problem
 \begin{equation}\label{e-averaged-problem}
 \inf_{u(\cdot), y(\cdot)}G(z(T)):= G^*,
 \end{equation}
 where $inf$ is taken over measurable functions $(u(\cdot), y(\cdot))\in \R^k\times \R^m$ and the corresponding solutions $z(\cdot) $ of the system
 \begin{equation}\label{e-system-free}
 \frac{dz(t)}{dt}=g(u(t),y(t),z(t))  ,   \ \ \ \ z(0)=z_0
 \end{equation}
 (both $u(\cdot)$ and $y(\cdot)$ are playing the role of controls in this system).

\begin{corollary}\label{Cor-OV-equality}
The optimal value in (\ref{e-SP-problem}) is equal to the optimal value in (\ref{e-averaged-problem}):
\begin{equation}\label{e-OV-equality}
  G_{\ve}^*= G^* \ \ \ \ \forall \ \ve >0.
 \end{equation}
\end{corollary}
\begin{proof}
Denote by $\mathcal{Z}_T^0$ the set of solutions of system (\ref{e-system-free}) considered on the interval $[0,T]$. As can be readily seen, the following inclusions are valid:
$$
\mathcal{Z}_T(\ve)\subset \mathcal{Z}_T^0\subset \mathcal{Z}_T.
$$
Therefore, by (\ref{eq-EQ}),
$$
{\rm cl}(\mathcal{Z}_T(\ve))= {\rm cl}( \mathcal{Z}_T^0)= \mathcal{Z}_T.
$$
The latter implies (\ref{e-OV-equality}).
\end{proof}

\section{Proof of the Main Result}

Consider the
system
 \begin{equation}\label{e:associate-sys}
\frac{dy(\tau)}{d\tau} = Ay(\tau)+Bu(\tau).
\end{equation}
Note that this system  looks similar to the fast subsystem (\ref{e:fast}) but, in contrast
to the latter, it evolves in the time
scale $\tau=\frac{t}{\ve}$ ((\ref{e:associate-sys}) will be referred to as the  {\em  associated system}). The following lemma is the key element of the subsequent analysis.

\begin{lemma}\label{Lem-control}
For any $y', y''\in \R^m$ there exists a control $u(\cdot)$ that steers the associated system (\ref{e:associate-sys}) from $y'$  to $y''$ in arbitrarily short period of time.
\end{lemma}
\begin{proof}
The proof of the lemma follows a standard argument, and we give it at the end of this section.
\end{proof}

Let $y(\tau, u(\cdot),y)$ stand for the solution of the associated system (\ref{e:associate-sys}) obtained with a control $u(\tau)$ and the initial condition $y(0)=y$. For a fixed $z$, denote by $V(S,y,z)$ the set of the time averages:
\begin{equation} \label{e:V(S)}
V(S,y ,z) :=
\bigcup_{u(\cdot)}  {\frac{1}{S}
\int_0^S g\brou{u(\tau),y(\tau, u(\cdot),y),z}d\tau},\quad z={\rm const},
\end{equation}
where the union is taken over all controls (that is, over all measurable functions $u(\cdot)$ taking values in $\R^k$). The set of the time averages similar
to (\ref{e:V(S)}) was introduced/used in \cite{Dont-Don}, \cite{Gai0},  \cite{Gai1}, \cite{Gra}, \cite{QW},  where it was shown that, as $S$ tends to infinity, it converges in the  Hausdorff metrics to a convex and compact set (provided that the associated system satisfies certain controllability or stability conditions). It was also shown
 that it is this set that defines the right-hand-side of the differential  inclusion, the solutions of which approximate the dynamics of the slow components of the SP control system.

The lemma below establishes that in the case under consideration the set ${\rm cl}(V(S,y_0 ,z))$ (the closure of the set $V(S,y_0 ,z) $) is equal to a convex and compact subset of $\R^m$ for any $S>0$, and it also provides an explicit representation for this set.

\begin{lemma}\label{Th-1}
For any $\ z\in \R^n$, any  $y  \in \R^m$, and any $S>0$, the following equality is valid:
\begin{equation} \label{e:Z4}
{\rm cl}\,(V(S,y ,z))=\bar{\rm co}\, (g(\R^k,\R^m,z)).
\end{equation}
\end{lemma}

\begin{proof}
Since $z$ is a constant parameter, we will suppress it in the notation, that is, we will write (\ref{e:V(S)}) in the form
\begin{equation} \label{e:V(S)-no-z}
V(S,y) :=
\bigcup_{u(\cdot)}  {\frac{1}{S}
\int_0^S g\brou{u(\tau),y(\tau , u(\cdot), y)}d\tau}.
\end{equation}
The inclusion
\begin{equation} \label{Z40}
V(S,y)\subset \bar{\rm co}\,(g(\R^k,\R^m)),
\end{equation}
follows from the definition of the integral. Indeed, take $v\in V(S,y_0)$, then there exists $u(\cdot) $ such that $v=\frac{1}{S}
\int_0^S g\brou{u(\tau),y(\tau, u(\cdot), y)}d\tau$. If the integrand was a simple function, that is, it was a finite sum of indicator functions of Borel measurable sets, then it would be clear that
$$
\frac{1}{S}\int_0^S g\brou{u(\tau),y(\tau , u(\cdot), y)}d\tau\in {\rm co}\,(g(\R^k,\R^m)),
$$
that is, $v\in {\rm co}\,g(\R^r,\R^m)$.
For an arbitrary Borel $g$, the inclusion $v\in \bar{\rm co}\,g(\R^r,\R^m)$
follows from the definition of Lebesgue integral as the limit of integrals of simple functions.

Inclusion (\ref{Z40}) implies that
$$
{\rm cl}\,(V(S,y))\subset \bar{\rm co}\,(g(\R^k,\R^m)).
$$
Let us show that
\begin{equation}\label{Z4}
{\rm co}\,(g(\R^k,\R^m))\subset {\rm cl}\,(V(S,y)).
\end{equation}
(This will, obviously, imply that $\bar{\rm co}\,(g(\R^k,\R^m))\subset {\rm cl}\,(V(S,y)) $.)
Take $v\in {\rm co}\,( g(\R^k,\R^m))$. Then, due to the Caratheodori Theorem,
$$
v=\sum_{i=1}^{l} \lambda_i g(u_i,y_i), \ \ \ \ l\leq k+m+1
$$
for some $(u_i,y_i) \in \R^k \times \R^m$ and some $\lambda_i>0$ with $\ \sum_{i=1}^{l} \lambda_i=1$. Let
$$
S_0=0,\; \ \ \ S_j:=S\sum_{i=1}^{j} \lambda_i,\ \  j=1,\dots, l,\;
$$
which implies that
\begin{equation} \label{Z2}
\lambda_j=\frac{S_{j}-S_{j-1}}{S}.
\end{equation}
Let $\delta > 0$ be arbitrary small
and let $\hat\tau > 0$ be defined by the equation
\begin{equation}\label{e-tau}
\hat\tau := \frac{\delta}{M_g},
\end{equation}
where $M_g$ is defined in (\ref{e-g-bound}). Note that from (\ref{e-tau}) it follows that the solution $y(\tau, u,y)$ of the associated system (\ref{e:associate-sys}) obtained with the  constant valued control $u(\tau)\equiv u$ and with the initial condition $y(0)=y$ satisfies the inequality
\begin{equation}\label{e-tau-1}
\max_{\tau\in [0,\hat\tau]}\|y(\tau, u, y)-y\|\leq \delta\ \ \ \ \forall \ (u,y)\in \R^k\times\R^m.
\end{equation}
Let us construct a control-state process  $(u(\cdot),y(\cdot))$ of the  system   (\ref{e:associate-sys}) on the interval $[0,S]$ such that, for all $j=1,\dots, l$, the following inequalities hold true (for brevity, the solution of (\ref{e:associate-sys}) is denoted below as $y(\tau)$ instead of $y(\tau, u(\cdot),y)$):
$$
\left|\frac{1}{S_{j}-S_{j-1}}\int_{S_{j-1}}^{S_{j}} g(u(\tau),y(\tau))\,d\tau-g(u_j, y_j)\right|\le K\delta
$$
for some constant $K$ (recall that $S_0:= 0$ and $S_l=S$).

Consider the interval $[S_0,S_1]$. Assuming (without loss of generality) that $\delta <S_1 $, define the control $u(\cdot)$
on the interval $[0,\delta]\subset [0,S_1]$ in such a way that $y(\delta)=y_1$. (This is possible due to Lemma \ref{Lem-control}.) Set $\tau_{11}:=\delta $ and $\tau_{12}:=\delta +\hat \tau .$ Extend the definition of the control $u(\cdot)$ by taking it to be equal to $u_1$ on the interval   $(\tau_{11}, \tau_{12}] $ if
$\tau_{12}< S_1 $. In case $\tau_{12}\geq S_1 $, take $u(\cdot)$ to be equal to $u_1$  on the interval $ (\tau_{11}, S_1] $. The control $u(\cdot)$ will be, thus, defined on $[S_0,S_1$]. Note that, by (\ref{e-tau-1}),
$$
\max_{\tau\in [\tau_{11}, \tau_{12}]}||y(\tau)-y_1||\leq \delta \ \ {\rm if} \ \ \tau_{12}< S_1
$$
and
$$
\max_{\tau\in [\tau_{11}, S_1]}||y(\tau)-y_1||\leq \delta \ \ {\rm if} \ \ \tau_{12}\geq S_1 .
$$

If $\tau_{12}< S_1 $, set  $\tau_{13}:= \min \{\tau_{12}   +\frac{\delta}{2},\tau_{12}   + \frac{S_1-\tau_{12}}{2}\} $ and $\tau_{14}:= \tau_{13} +\hat\tau $. Extend the definition of the control $u(\cdot)$ to the interval $(\tau_{12},\tau_{13}]$
in such a way that the corresponding solution  $y(\tau)$
of the system (\ref{e:associate-sys}) satisfies  the equation $y(\tau_{13})=y_1$. (Again, this is possible due to Lemma \ref{Lem-control}.)
Also, extend the definition of the control $u(\cdot)$ further by taking it to be equal to $u_1$ on the interval $(\tau_{13}, \tau_{14}]$ if
$\tau_{14}< S_1 $. In case $\tau_{14}\geq S_1 $, take $u(\cdot)$ to be equal to $u_1$  on the interval $ (\tau_{13}, S_1] $. The control $u(\cdot)$ will be, thus, defined on $[S_0,S_1$]. By (\ref{e-tau-1}),
$$
\max_{\tau\in [\tau_{11}, \tau_{12}]\cup [\tau_{13}, \tau_{14}]}||y(\tau)-y_1||\leq \delta \ \ {\rm if} \ \ \tau_{14}< S_1
$$
and
$$
\max_{\tau\in [\tau_{11}, \tau_{12}]\cup [\tau_{13}, S_1]}||y(\tau)-y_1||\leq \delta \ \ {\rm if} \ \ \tau_{14}\geq S_1 .
$$

In the general case (for an arbitrary small $\delta$), one can proceed in a similar way to:

(i) Define the sequence of moments of time $\tau_{11}, \tau_{12}, \cdots , \tau_{1\bar s}, \tau_{1\bar s+1}$, where $\bar s$ is odd,
$\tau_{\bar s+1} + \bar\tau \geq S_1$, and, for any odd $s$ such that $ 3\leq s < \bar s$,
\begin{equation} \label{e:Z1-0}
\tau_{1 s}:= \min \{\tau_{1s-1}   +\frac{\delta}{2^{s-2}},\tau_{1s-1}   + \frac{S_1-\tau_{1s-1}}{2}\}, \ \ \ \ \tau_{1 s+1}:= \tau_{1 s} +\hat\tau < S_1;
\end{equation}

and

(ii) Construct the control $u(\cdot)$ that along with the corresponding solution $y(\cdot)$ of (\ref{e:associate-sys}) satisfy the relations
\begin{equation} \label{e:Z1}
u(\tau)\equiv u_1, \; ||y(\tau)-y_1||\le \delta \; \forall \tau\in  D:= (\tau_{11},\tau_{12}] \cup(\tau_{13},\tau_{14}] \cup \ldots \cup (\tau_{1\bar s-2},\tau_{1\bar s-1}]\cup (\tau_{1\bar s}, S_1],
\end{equation}
with $y(\tau_{1,s})=y_1$ for all  $s=1,3,\ldots, \bar s$.

By construction (see (\ref{e:Z1-0}) and (\ref{e:Z1})), the Lebesgue measure of the set $[0,S_1]\setminus D$ (which we denote as $meas([0,S_1]\setminus D)$ satisfies the inequality
$$
meas([0,S_1]\setminus D)= \tau_{11} + (\tau_{13}-\tau_{12})+\cdots + (\tau_{1\bar s}-\tau_{1\bar s-1})\leq  \delta + \frac{\delta}{2}+\cdots + \frac{\delta}{2^{\bar s-2}}
\leq 2\delta.
$$
Therefore,
$$
\left|\frac{1}{S_1}\int_{[0,S_1]\setminus D} (g(u(\tau),y(\tau))-g(u_1,y_1))\,d\tau\right|\leq \frac{2M_g}{S_1} meas ([0,S_1]\setminus D)\leq \frac{4M_g}{S_1}\delta ,
$$
where $M_g$ is as in (\ref{e-g-bound}).
Also, due to (\ref{e:Z1}),
$$
\left|\frac{1}{S_1}\int_{D} \big(g(u(\tau),y(\tau))\, -g(u_1,y_1)\big)d\tau\right|=\left|\frac{1}{S_1}\int_{D} g(u_1,y(\tau))\,d\tau-g(u_1,y_1)\right|\leq L\delta,
$$
where $L$ is a Lipschitz constant of $g$. Consequently,
\begin{equation*}
\begin{aligned}
&\left|\frac{1}{S_1}\int_0^{S_1} g(u(\tau),y(\tau))\,d\tau-g(u_1,y_1)\right|=
\left|\frac{1}{S_1}\int_0^{S_1} (g(u(\tau),y(\tau))-g(u_1,y_1))\,d\tau\right|\le\\
&\left|\frac{1}{S_1}\int_{[0,S_1]\setminus D} (g(u(\tau),y(\tau))-g(u_1,y_1))\,d\tau\right|+
\left|\frac{1}{S_1}\int_{D} (g(u(\tau),y(\tau))-g(u_1,y_1))\,d\tau\right|\le \\
&\left(\frac{4M_g}{ S_1}+L\right)\delta.
\end{aligned}
\end{equation*}
Continuing similarly on each subinterval $[S_{j-1},S_j]$, we construct the control-state  process $(u(\cdot),y(\cdot))$ on $[0,S]$ such that for all $j=1,\dots,l$
\begin{equation}\label{Z3}
\left|\frac{1}{S_{j}-S_{j-1}}\int_{S_{j-1}}^{S_{j}} g(u(\tau),y(\tau))\,
d\tau-g(u_j,y_j)\right|\le \left(\frac{4M_g}{S_{j}-S_{j-1}}+L\right)\delta.
\end{equation}
For this process, taking into account (\ref{Z2}) and (\ref{Z3}), we have
\begin{equation*}
\begin{aligned}
&\left|\frac{1}{S}\int_{0}^{S} g(u(\tau),y(\tau))\,d\tau-v\right|\\
=&\left|\sum_{i=1}^{l}\frac{S_{i}-S_{i-1}}{S}\frac{1}{S_{i}-S_{i-1}} \int_{S_{i-1}}^{S_{i}} g(u(\tau),y(\tau))\,d\tau-
\sum_{i=1}^{l}\lambda_ig(u_i,y_i)\right|\\
=&\sum_{i=1}^{l}\lambda_i\left|\frac{1}{S_{i}-S_{i-1}}\int_{S_{i-1}}^{S_{i}}g(u(\tau),y(\tau))\,d\tau-g(u_i,y_i)\right|\\
\le &\,\delta\sum_{i=1}^{l}\lambda_i\left(\frac{4M_g}{S_{i}-S_{i-1}}+L\right)
\le\delta\left(\frac{4M_g}{\min_{i}\{S_{i}-S_{i-1}\}}+L\right).
\end{aligned}
\end{equation*}
Since $\delta$ is arbitrarily small, this implies that $v\in {\rm cl}\,V(S,y)$, which implies (\ref{Z4}). The lemma is proved.
\end{proof}

\begin{lemma}\label{Prop-Lipschitz-elementary}
The multivalued function $\ V(z)$ defined in accordance with (\ref{eq-RHS}) is Lipschitz continuous. That is,
\begin{equation}\label{e:V-Lipschitz-continuity}
d_{H}(V(z'),V(z'')) \leq L||z'-z''||  \ \ \ \  \forall z', z'',
\end{equation}
where $d_{H}$ stands for the Hausdorff distance between sets and $L$ is the Lipschitz constant of $g(u,y,z)$ in $z$.
\end{lemma}

\begin{proof} Let $z', z''$ be arbitrary elements of $\R^n$.
Take some  $S>0$ and choose an arbitrary element $\ v' $ from $\ V(z',S,y_0)$. From the  definition of $\ V(z',S,y_0)\ $  it follows that there exists a control $\ u(\cdot) $ such that
$\
v' = \frac{1}{S}\int_0^S g(z',u(\tau), y(\tau))d\tau.
$
Define $\ v'' $ by
$$
v'':= \frac{1}{S}\int_0^S g(z'',u(\tau), y(\tau))d\tau \in V(z'',S,y_0).
$$
We have
$$
||v'-v''||\leq \frac{1}{S}\int_0^S||g(z',u(\tau), y(\tau))-g(z'',u(\tau), y(\tau))||d\tau
$$
$$
\leq L||z'-z''|| \ \ \ \ \Rightarrow \ \ \ \ \ d(v', V(z'')) \leq L||z'-z''||,
$$
where $d(v,V)$ stands for the distance from a vector $v$ to a set $V$ ($d(v,V):= \inf_{w\in V}||v-w|| $).
Since $\ v'$ is an arbitrary element of $\ V(z',S,y_0)$, it implies that
$$\ \sup_{v\in V(z',S,y_0)}d(v, V(z'',S,y_0))\leq L||z'-z''||.$$
Since $V(z)={\rm cl}\,V(S,y_0,z)$, we conclude that
$$
\sup_{v\in V(z')}d(v, V(z''))\leq L||z'-z''||.
$$
Similarly, it is established that
$ \ $$\ \sup_{v\in V(z'')}d(v, V(z'))\leq L||z'-z''||.$
\qquad\end{proof}

\medskip

{\em  Proof of Theorem \ref{Th-main}}.
Due to (\ref{eq-INCL}), to prove the theorem it is sufficient to show that
$$
\mathcal{Z}_T \subset {\rm cl}(\mathcal{Z}_T(\ve))  \ \ \ \ \forall \ \ve >0.
$$
This, in turn, will be shown, if we establish that,
for any $\delta>0$,
corresponding to any solution $z(t)$ of the differential inclusion  (\ref{eq-DI}), there exists  a control $u_{\ve}(t)$, which, being used in the system (\ref{e:fast})-(\ref{e:slow}), generates the solution $(y_{\ve}(t), z_{\ve}(t))$ that satisfies the inequality
\begin{equation}\label{e:Th2.3-3-extra}
\max_{t\in [0,T]}||z_{\ve}(t)- z(t)|| \leq \delta .
\end{equation}
To prove the latter statement, take an arbitrary solution $z(t)$ of (\ref{eq-DI}), choose an arbitrary $\delta > 0$ and construct the control $u_{\ve}(t)$ that insures the validity of (\ref{e:Th2.3-3-extra}).

Take  $S\in (0, \frac{T}{\ve}]$ and partition the interval $\ [0,T]\ $ by the points
\begin{equation}\label{e:partition}
t_l := l \left(\ve S\right) \ , \ \ \ \ l = 0,1,..., N_{\ve S}:= \left\lfloor \frac{T}{\ve S} \right\rfloor ,
\end{equation}
 where  $\lfloor\cdot\rfloor $ is the floor function (that is, for any real $x$, $\lfloor x\rfloor $ is the maximal integer that is less or equal than $x$).

 Let $v_0$ be the projection of the vector
$\
\left(\ve S\right)^{-1}\int_{0}^{t_{1}}\frac{dz(t)}{dt}dt \
$
onto the set $V(z_0)$. That is,
$$
v_0:= {\rm argmin}_{v\in V\left(z_0\right)} \left\{\left\| \left(\ve S\right)^{-1}\int_{0}^{t_{1}}\frac{dz(t)}{dt}dt -v \right\| \right\}.
$$
By Lemma \ref{Th-1}, ${\rm cl}\,V(S,y_0,z_0)=V(z_0)$. Therefore, there exists a control
$u_0(\tau)$   such that, being used in the associated system  (\ref{e:associate-sys}) on the interval $\ [0, \frac{t_{1}}{\ve}], \ $
ensures that
$$
\left\|v_0-\frac{1}{S}\int_{0}^{\frac{t_{1}}{\ve}}
g(u_0(\tau), y_0(\tau), z_0) d\tau\right\|\le \ve S,
$$
where $y_0(\tau):= y(\tau , u_0(\cdot),y_0)  $.

Define the control $u_{\ve}(t) $  on the interval $[0,t_1)$ by the equation
$$
 u_{\ve}(t):= u_0\left(\frac{t}{\ve}\right)\ \ \forall \ t\in [0,t_1)
$$
and show how this definition can be extended to the interval $[0,T] $.

Assume  that the control $u_{\ve}(t) $ has been defined on the interval $[0,t_l)$ and denote by $(y_{\ve}(t), z_{\ve}(t))$   the corresponding solution of  the system (\ref{e:fast})-(\ref{e:slow}) on this interval. Extend the definition of $u_{\ve}(t) $ to the interval
$[0,t_{l+1}) $ by following the steps:

{  (a)} Define  $v_l\ $ as the projection of the vector
$\
\left(\ve S\right)^{-1}\int_{t_l}^{t_{l+1}}\frac{dz(t)}{dt}dt\
$
onto the set $V(z_{\ve}(t_l))$. That is,
\begin{equation}\label{e-proj-1}
v_l:= {\rm argmin}_{v\in V(z_{\ve}(t_l))} \left\{\left\| \left(\ve S\right)^{-1}\int_{t_l}^{t_{l+1}}\frac{dz(t)}{dt}dt -v \right\| \right\}.
\end{equation}

{  (b)} Define a control
$u_l(\tau)$  such that, being used in the associated system  (\ref{e:associate-sys}) on the interval $\ [\frac{t_l}{\ve}, \frac{t_{l+1}}{\ve}], \ $
ensures that
\begin{equation}\label{e:proof-10-1-extra}
\left\|v_l-\frac{1}{S}\int_{\frac{t_l}{\ve}}^{\frac{t_{l+1}}{\ve}}
g(u_l(\tau),  y_l(\tau) , z_{\ve}(t_l)) d\tau\right\|\le \ve S ,
\end{equation}
where $ y_l(\tau) $ is the solution of the associated system (\ref{e:associate-sys})   obtained with the   control $u_l(\tau)$ and with the initial condition $y_l(\frac{t_l}{\ve})= y_{\ve}(t_l) $.   The existence of such control follows the fact that, by Lemma \ref{Th-1},  ${\rm cl}\,V(S,y_{\ve}(t_l),z_{\ve}(t_l)) = V(z_{\ve}(t_l)) $.

{  (c)} Define the control $u_{\ve}(t)$ on the interval $[t_l, t_{l+1})$ as equal to $u_l(\frac{t}{\ve})$. That is,
\begin{equation}\label{e:z-aaproximating control}
\ u_{\ve}(t):= u_l\left(\frac{t}{\ve}\right)\ \ \forall \ t\in [t_l, t_{l+1}).
\end{equation}

Proceeding in a similar way, one can extend the  definition of the control
$u_{\ve}(t)$  to the interval $[0,t_{N_{\ve S}})$. On the interval $[t_{N_{\ve S}}, T] $, the control  $u_{\ve}(t)$ can be taken to be equal to an arbitrary element $u$. Denote by $(y_{\ve}(t), z_{\ve}(t))$   the corresponding solution of the system (\ref{e:fast})-(\ref{e:slow}) on the interval $[0,T]$.

Denote
$$
(\bar y_{\ve}(\tau), \bar z_{\ve}(\tau)):= (y_{\ve}(\ve\tau), z_{\ve}(\ve\tau))\ \ \ \ \forall \ \tau\in [0, \frac{T}{\ve}].
$$
Then
$$
(\bar y_{\ve}(\tau_l), \bar z_{\ve}(\tau_l)) = (y_{\ve}(t_l), z_{\ve}(t_l)) \ \ \ \ \forall \ l=0,1,..., N_{\ve},
$$
where $\ \tau_l= \frac{t_l}{\ve}=lS$ (see (\ref{e:partition})), and
\begin{equation}\label{e:z-epsSeps-estimate}
\max_{\tau\in [\tau_l,\tau_{l+1}] }||\bar z_{\ve} (\tau)- \bar z_{\ve} (\tau_l)||
= \max_{t\in [t_l,t_{l+1}] }|| z_{\ve} (t)- z_{\ve} (t_l)||
\leq M_g\ve S ,
\end{equation}
where $M_g$ is defined in (\ref{e-g-bound}).
Also notice that
\begin{equation} \label{e:est-z-2-1}
 \max_{t \in [t_l, t_{l+1}]}||z(t)-z(t_l)|| = \max_{\tau \in [\tau_l, \tau_{l+1}]}||z(\ve\tau)-z(t_l)|| \leq M_g \ve S,
\end{equation}
since
$$
\max\{||v|| \ | \  v\in V(z) ,\
z \in Z \} \leq M_g.
$$

Let us verify that thus defined control ensures the validity of (\ref{e:Th2.3-3-extra}).
Subtracting the equation
$$
z(t_{l+1}) = z(t_l) +
\int_{t_l}^{ t_{l+1}}\frac{dz(t)}{dt}dt
$$
from the equation
$$
z_{\ve}(t_{l+1}) = z_{\ve}(t_l) +
\int_{t_l}^{ t_{l+1}}g(u_{\ve}(t), y_{\ve}(t),z_{\ve}(t))dt,
$$
one obtains, taking into account (\ref{e:z-aaproximating control}), that
\begin{equation}\label{e:proof-2-17-1}
\begin{aligned}
&||z_{\ve}(t_{l+1})- z( t_{l+1})|| \leq
||z_{\ve}(t_{l})- z( t_{l})||\\
&+ \ve \int_{\tau_l}^{\tau_{l+1}}||g(u_l(\tau),\bar y_{\ve}(\tau), \bar z_{\ve}(\tau))
-g(u_l(\tau),\bar y_{\ve}(\tau), \bar z_{\ve}(\tau_l)) || d\tau\\
&+ \ve S \left\|\frac{1}{S}\int_{\tau_l}^{\tau_{l+1}}
g(u_l(\tau),\bar y_{\ve}(\tau), \bar z_{\ve}(\tau_l))d\tau  - v_l \right\|
+ \ve S \left\|v_l - \frac{1}{\ve S}\int_{t_l}^{ t_{l+1}}\frac{dz(t)}{dt}dt\right\|.
\end{aligned}
 \end{equation}
From Lipschitz continuity of $g(u,y,z)$ in $z$ and (\ref{e:z-epsSeps-estimate}) we obtain
\begin{equation}\label{e-estimate-aux-1}
\ve \int_{\tau_l}^{\tau_{l+1}}||g(u_l(\tau),\bar y_{\ve}(\tau), \bar z_{\ve}(\tau))
-g(u_l(\tau),\bar y_{\ve}(\tau), \bar z_{\ve}(\tau_l)) || d\tau \leq \ve S L \left(M_g \ve S\right).
\end{equation}
Due to (\ref{e:V-Lipschitz-continuity}) and (\ref{e:est-z-2-1}),
\begin{equation*}
\begin{aligned}
&\frac{dz(t)}{dt} \in V(z(t))\subset V(z(t_l)) + L||z(t)-z(t_l)||\bar B^n\\
&\subset V(z_{\ve}(t_l)) + L(||z(t)-z(t_l)||+ ||z(t_l)- z_{\ve}(t_l) ||)\bar B^n\\
&\subset V(z_{\ve}(t_l)) + (LM_g \ve S+ L||z(t_l)- z_{\ve}(t_l) ||) \bar B^n \ \ \ \forall \ t \in [t_l,t_{l+1}],
\end{aligned}
 \end{equation*}
where $\bar B^n$ is the closed unit ball in $R^n$. Consequently,
$$
\frac{1}{\ve S}\int_{t_l}^{t_{l+1}}\frac{dz(t)}{dt}dt \in
 V(z_{\ve}(t_l)) + (LM \ve S+ L||z(t_l)- z_{\ve}(t_l) ||)\bar B^n,
$$
and, by (\ref{e-proj-1}),
\begin{equation*}
\begin{aligned}
&\left\| \frac{1}{\ve S}\int_{t_l}^{ t_{l+1}}\frac{dz(t)}{dt}dt - v_l\right\|= d\left(\frac{1}{\ve S}\int_{t_l}^{ t_{l+1}}\frac{dz(t)}{dt}dt, V(z_{\ve}(t_l))\right)\\
&\leq LM_g \ve S+ L||z(t_l)- z_{\ve}(t_l) ||.
\end{aligned}
 \end{equation*}
Using the latter and  (\ref{e-estimate-aux-1}), (\ref{e:proof-10-1-extra}), one can obtain from (\ref{e:proof-2-17-1})
\begin{equation*}
\begin{aligned}
&||z_{\ve}(t_{l+1})- z( t_{l+1})|| \leq
||z_{\ve}(t_{l})- z( t_{l})|| +  \ve S L \left(M_g \ve S \right)+(\ve S)^2\\
&+ LM_g (\ve S)^2 +
L\ve S||z(t_l)- z_{\ve}(t_l) ||
=(1+L\ve S)||z(t_l)- z_{\ve}(t_l) ||+(\ve S)^2(2LM_g+1)\\
&\leq \left(1+\frac{LT}{N_{\ve S}}\right)||z(t_l)- z_{\ve}(t_l) ||+\left(\frac{T}{N_{\ve S}}\right)(\ve S) (2LM_g+1),
\end{aligned}
 \end{equation*}
where the last inequality follows from the definition of $N_{\ve S}$ (see (\ref{e:partition})).
Using an argument similar to that of Gronwall's lemma
(see, e.g., Proposition 5.1 in \cite{Gai1}), one  can now derive that
 \begin{equation}\label{e-estimate-aux-2}
||z_{\ve}(t_{l})- z( t_{l})||\leq L^{-1}e^{LT}(2LM_g+1) \ve S, \ \ \ l=0,1,..., N_{\ve S}.
\end{equation}
Estimate (\ref{e-estimate-aux-2}), along with (\ref{e:z-epsSeps-estimate}) and (\ref{e:est-z-2-1}), imply that
$$
||z_{\ve}(t)- z( t)||\leq L^{-1}e^{LT}(2LM_g+1) \ve S+2M_g\ve S  \ \  \; \forall t\in [0,T].
$$
The right-hand-side in the expression above tends to zero with $S$ tending to zero. Therefore, it can be made less or equal than $\delta$ if $S$ is chosen small enough.
This proves (\ref{e:Th2.3-3-extra}).
\endproof

{\em  Proof of Lemma \ref{Lem-control}.}
By Cauchy formula, the lemma will be proved if we show that, for any $\tau >0$, there exists a control $u(\cdot)$ such that
\begin{equation}\label{e-L1-1}
y''=e^{A\tau}y'+\int_0^{\tau}e^{A(\tau-s)}Bu(s)\,ds\ \ \ \ \Leftrightarrow\ \ \ \ e^{-A\tau}y'' -y' = \int_0^{\tau}e^{-As}Bu(s)\,ds .
\end{equation}
Take  $u(s)=B^Te^{-A^Ts}\xi$ and show that there exists $\xi\in\R^m$ such that these equalities are satisfied.
To this end, substitute $u(s)=B^Te^{-A^Ts}\xi$ into the second equality in (\ref{e-L1-1}) to obtain
$$
e^{-A\tau}y'' -y' = \left(\int_0^{\tau}e^{-As}BB^Te^{-A^Ts}\,ds\right)\xi,
$$
therefore,
$$
\left(\int_0^{\tau}e^{-As}BB^Te^{-A^Ts}\,ds\right)^{-1}\Big(e^{-A\tau}y'' -y'\Big)=\xi ,
$$
where the latter implication is valid if the matrix $\left(\int_0^{\tau}e^{-As}BB^Te^{-A^Ts}\,ds\right)$ is non-singular.

Thus, a sufficient condition for (\ref{e-L1-1}) to be valid is non-singularity of the matrix $\left(\int_0^{\tau}e^{-As}BB^Te^{-A^Ts}\,ds\right)$.
Since this matrix
$\left(\int_0^{\tau}e^{-As}BB^Te^{-A^Ts}\,ds\right)$ is non-negative definite, to show that it is non-singular, one needs to show that it is positive definite.
Assume the contrary, that is, there exists a nonzero vector $v$ and time $\tau$ such that
$$
v^T\left(\int_0^{\tau}e^{-As}BB^Te^{-A^Ts}\,ds\right)v=0.
$$
This is possible only if
$$
v^Te^{-As}BB^Te^{-A^Ts}v\equiv 0\quad \hbox{for all }s\in [0,\tau].
$$
The latter identity implies  that
$$
v^Te^{-As}B\equiv 0\quad \hbox{for all } s\in [0,\tau].
$$
Differentiating this identity with respect to $s$ $j$ times, $j=0,\dots,m-1$ and plugging in $s=0$ we get
$$
v^TB=0,\,v^TAB=0,\,\dots,v^TA^{m-1}B=0,
$$
or, in a matrix form,
$$
v^T[B, AB,\dots,A^{m-1}B]=0.
$$
This implies that the rank of the matrix $[B, AB,\dots,A^{m-1}B]$ is less than $m$, which is a contradiction to our assumption
(\ref{e-rank}).
\endproof

\end{document}